\documentclass[12pt]{amsart}
\usepackage{amscd}
\usepackage{graphicx}
\usepackage{caption}
\usepackage{amsthm}
\usepackage{amsfonts}
\usepackage{amsmath}
\usepackage{amssymb}
\def\frk{\frak}               

\def\Phi{{\frk n}}
\def\Phi{{\frk N}}
\def\opn#1#2{\def#1{\operatorname{#2}}} 
\opn\chara{char} \opn\length{\ell} \opn\pd{pd} \opn\rk{rk}
\opn\projdim{proj\,dim} \opn\injdim{inj\,dim} \opn\rank{rank}
\opn\depth{depth} \opn\grade{grade} \opn\height{height}
\opn\embdim{emb\,dim} \opn\codim{codim}

\opn\Tr{Tr} \opn\bigrank{big\,rank}
\opn\superheight{superheight}\opn\lcm{lcm}
\opn\trdeg{tr\,deg}
\opn\reg{reg} \opn\lreg{lreg} \opn\ini{in} \opn\lpd{lpd}
\opn\size{size}
\opn\div{div} \opn\Div{Div} \opn\cl{cl} \opn\Cl{Cl}
\opn\Spec{Spec} \opn\Supp{Supp} \opn\supp{supp} \opn\Sing{Sing}
\opn\Ass{Ass} \opn\Min{Min}
\opn\Ann{Ann} \opn\Rad{Rad} \opn\Soc{Soc}
\opn\Im{Im} \opn\Ker{Ker} \opn\Coker{Coker} \opn\Am{Am}
\opn\Hom{Hom} \opn\Tor{Tor} \opn\Ext{Ext} \opn\End{End}
\opn\Aut{Aut} \opn\id{id}

\opn\nat{nat}
\opn\pff{pf}
\opn\Pf{Pf} \opn\GL{GL} \opn\SL{SL} \opn\mod{mod} \opn\ord{ord}
\opn\Gin{Gin} \opn\Hilb{Hilb}
\opn\aff{aff} \opn\con{conv} \opn\relint{relint} \opn\st{st}
\opn\lk{lk} \opn\cn{cn} \opn\core{core} \opn\vol{vol}
\opn\link{link} \opn\star{star}
\opn\gr{gr}

\def\pot#1#2{#1[\kern-0.28ex[#2]\kern-0.28ex]}

\opn\dirlim{\underrightarrow{\lim}}
\opn\inivlim{\underleftarrow{\lim}}
\def\Implies{\ifmmode\Longrightarrow \else
        \unskip${}\Longrightarrow{}$\ignorespaces\fi}
\def\implies{\ifmmode\Rightarrow \else
        \unskip${}\Rightarrow{}$\ignorespaces\fi}
\def\iff{\ifmmode\Longleftrightarrow \else
        \unskip${}\Longleftrightarrow{}$\ignorespaces\fi}

\let\:=\colon
\newtheorem{Theorem}{Theorem}[section]
\newtheorem{Lemma}[Theorem]{Lemma}

\newtheorem{Proposition}[Theorem]{Proposition}

\newtheorem{Example}[Theorem]{Example}

\newtheorem{Definition}[Theorem]{Definition}

\opn\Syz{Syz} \opn\Im{Im} \opn\Ker{Ker} \opn\Coker{Coker}
\opn\Am{Am} \opn\Hom{Hom} \opn\Tor{Tor} \opn\Ext{Ext} \opn\End{End}
\opn\Aut{Aut} \opn\id{id}

\opn\nat{nat}
\opn\pff{pf}
\opn\Pf{Pf} \opn\GL{GL} \opn\SL{SL} \opn\mod{mod} \opn\ord{ord}
\opn\Gin{Gin}\opn\min{min}
\opn\Hilb{Hilb}\opn\adeg{adeg}\opn\std{std}\opn\ip{infpt}
\opn\Pol{Pol}\opn\sdepth{sdepth}\opn\infpt{infpt}
\opn\depth{depth}\opn\sqdepth{sqdepth}\opn{\Mon}{Mon}

\let\epsilon\varepsilon
\let\phi=\varphi
\let\kappa=\varkappa
\def\qed{\ifhmode\textqed\fi
      \ifmmode\ifinner\quad\qedsymbol\else\dispqed\fi\fi}
\def\textqed{\unskip\nobreak\penalty50
       \hskip2em\hbox{}\nobreak\hfil\qedsymbol
       \parfillskip=0pt \finalhyphendemerits=0}
\def\dispqed{\rlap{\qquad\qedsymbol}}

\opn\dis{dis}
\def\pnt{{\raise0.5mm\hbox{\large\bf.}}}

\opn\Lex{Lex}

\begin{document}
\title
{Spanning Simplicial Complexes of Uni-Cyclic Multigraphs}

\author{Imran Ahmed, Shahid Muhmood}

\address{COMSATS Institute of Information Technology, Lahore, Pakistan}
\email{drimranahmed@ciitlahore.edu.pk}
\address{COMSATS Institute of Information Technology, Lahore, Pakistan}
\email{shahid\_nankana@yahoo.com}
 \maketitle
\begin{abstract} A multigraph is a nonsimple graph which is permitted to have multiple edges, that is, edges that have the same end nodes. We introduce the concept of spanning simplicial complexes $\Delta_s(\mathcal{G})$ of multigraphs $\mathcal{G}$, which provides a generalization of spanning simplicial complexes of associated simple graphs.  We give first the characterization of all spanning trees of a uni-cyclic multigraph $\mathcal{U}_{n,m}^r$ with $n$ edges including $r$ multiple edges within and outside the cycle of length $m$. Then, we determine the facet ideal $I_\mathcal{F}(\Delta_s(\mathcal{U}_{n,m}^r))$ of
spanning simplicial complex $\Delta_s(\mathcal{U}_{n,m}^r)$ and its primary decomposition. The Euler characteristic is a well-known
topological and homotopic invariant to classify surfaces. Finally, we device a formula for Euler characteristic of spanning
simplicial complex $\Delta_s(\mathcal{U}_{n,m}^r)$.
 \vskip 0.4 true cm
 \noindent
\end{abstract}

\pagestyle{myheadings} \markboth{\centerline {\scriptsize Ahmed and
Muhmood}}
         {\centerline {\scriptsize Spanning Simplicial Complexes of Uni-Cyclic Multigraphs}}

\maketitle

\section{Introduction}

Let $\mathcal{G}=\mathcal{G}(V,E)$ be a multigraph  on the vertex set $V$ and edge-set $E$. A spanning tree of a multigraph $\mathcal{G}$ is a subtree of $\mathcal{G}$ that contains every vertex of $\mathcal{G}$. We represent the collection of all edge-sets of the spanning trees of a multigraph $\mathcal{G}$ by $s(\mathcal{G})$. The facets of spanning simplicial complex $\Delta_s(\mathcal{G})$ is exactly the edge set $s(\mathcal{G})$ of
all possible spanning trees of a multigraph $\mathcal{G}$. Therefore, the spanning simplicial complex
$\Delta_s(\mathcal{G})$ of a multigraph $\mathcal{G}$  is defined by
$$\Delta_s(\mathcal{G})=\langle F_k\,|\,F_k\in s(\mathcal{G})\rangle,$$
which gives a generalization of the spanning simplicial complex $\Delta_s(G)$ of an associated simple graph $G$. The spanning simplicial complex of a simple connected finite graph was firstly introduced by Anwar, Raza and Kashif in \cite{AZK}. Many authors discussed algebraic and combinatorial properties of spanning simplicial complexes of various classes of simple connected finite graphs, see for instance \cite{AZK}, \cite{KAR}, \cite{PLZ} and \cite{ZSG}.

Let $\Delta$ be simplicial complex of dimension $d$. We denote $f_i$ by the number of $i$-cells of simplicial complex $\Delta$.
Then, the Euler characteristic of $\Delta$ is given by
$$\chi(\Delta)=\sum_{i=0}^{d-1}(-1)^i f_i,$$
which is a well-known topological and homotopic invariant to classify surfaces,
see \cite{H} and \cite{R}.

The uni-cyclic multigraph $\mathcal{U}_{n,m}^r$ is a  connected graph
having $n$ edges including $r$ multiple edges within and outside the cycle of length
$m$. Our aim is to give some algebraic and topological
characterizations of spanning simplicial complex
$\Delta_s(\mathcal{U}_{n,m}^r)$.

In Lemma \ref{l1}, we give characterization of all spanning trees of
a uni-cyclic multigraph $\mathcal{U}_{n,m}^r$ having $n$ edges
including $r$ multiple edges and a cycle of length $m$. In
Proposition \ref{p1}, we determine the facet ideal
$I_\mathcal{F}(\Delta_s(\mathcal{U}_{n,m}^r))$ of spanning
simplicial complex $\Delta_s(\mathcal{U}_{n,m}^r)$ and its primary
decomposition. In Theorem \ref{t1}, we give a formula for Euler
characteristic of spanning simplicial complex
$\Delta_s(\mathcal{U}_{n,m}^r)$.

\section{Basic Setup}

A simplicial complex $\Delta$ on $[{n}]=\{1,\ldots, n\}$ is a
collection of subsets of $[n]$ satisfying the following
properties.\\
\noindent $(1)$ $\{j\}\in \Delta$ for all $j\in[n]$;\\
\noindent $(2)$ If $F\in \Delta$ then every subset of $F$ will
belong to $\Delta$ (including empty set).

The elements of $\Delta$ are called faces of $\Delta$ and the
dimension of any face $F\in\Delta$ is defined as $|F|-1$ and is
written as $\dim F$, where $|F|$ is the number of vertices of $F$.
The vertices and edges are $0$ and $1$ dimensional faces of $\Delta$
(respectively), whereas, $dim\ \emptyset= -1$. The maximal faces of
$\Delta$ under inclusion are said to be the facets of $\Delta$. The
dimension of $\Delta$ is denoted by $\dim\Delta$ and is defined by
$\dim\Delta=max\{\dim F\ |\ F\in\Delta\}$.

If $\{F_1,\ldots ,F_q\}$
is the set of all the facets of $\Delta$, then $\Delta=<F_1,\ldots
,F_q>$. A simplicial complex $\Delta$ is said to be pure, if all its
facets are of the same dimension.

A subset $M$ of $[n]$ is said to be a vertex
cover for $\Delta$ if $M$ has non-empty intersection with every
$F_k$. $M$ is said to be a minimal vertex cover for $\Delta$ if no
proper subset of $M$ is a vertex cover for $\Delta$.

\begin{Definition}
Let $\mathcal{G}=\mathcal{G}(V,E)$ be a multigraph  on the vertex set $V$ and edge-set $E$. A spanning tree of a multigraph $\mathcal{G}$ is a subtree of $\mathcal{G}$ that contains every vertex of $\mathcal{G}$.
\end{Definition}

\begin{Definition}
Let $\mathcal{G}=\mathcal{G}(V,E)$ be a multigraph  on the vertex set $V$ and edge-set $E$. Let $s(\mathcal{G})$ be the edge-set of all possible spanning trees of $\mathcal{G}$. We define a simplicial complex $\Delta_s(\mathcal{G})$ on $E$ such that the facets of $\Delta_s(\mathcal{G})$ are exactly the elements of $s(\mathcal{G})$, we call $\Delta_s(\mathcal{G})$ as the spanning simplicial complex of $\mathcal{G}$ and given by
$$\Delta_s(\mathcal{G})=\langle F_k\,|\,F_k\in s(\mathcal{G})\rangle.$$
\end{Definition}

\begin{Definition}
A uni-cyclic multigraph $\mathcal{U}_{n,m}^r$ is a connected graph
having $n$ edges including $r$ multiple edges within and outside the cycle of length
$m$.
\end{Definition}

Let $\Delta$ be simplicial complex of dimension $d$. Then, the chain complex $C _{\ast}
({\Delta})$ is given by
\begin{center}
$0 \rightarrow C_{d}(\Delta) ^{
\underrightarrow{\partial_{d}}} C_{d-1}(\Delta)^{\underrightarrow{\partial_{d-1}}} ... ^{\underrightarrow{\partial_{2}}}C_{1}({\Delta})^{
\underrightarrow{\partial_{1}}}C_{0}({\Delta})^{\underrightarrow\partial_{0}} 0$.
\end{center}
Each $C_{i}({\Delta})$ is a free abelian group of rank $f_{i}$.
The boundary homomorphism $\partial_{d}: C_{d}(\Delta)\rightarrow C_{d-1}(\Delta)$ is defined by
\begin{center}
$\partial_{d}(\sigma_{\alpha}^{d})=\sum_{i=0}^{d}(-1)^{i}\sigma_{\alpha}^{d}|_{[{v_{0},v_{1},...,\hat{v{i}},...,v_{d}}]}.$
\end{center}
Of course,
$$H_{i}(\Delta) = Z_{i}(\Delta) / B_{i}(\Delta) = Ker\partial_{i}/Im \partial_{i+1},$$
where $Z_i(\Delta)=Ker\,\partial_{i}$ and $B_i(\Delta)=Im\, \partial_{i+1}$ are the groups of simplicial $i$-cycles and simplicial $i$-boundaries, respectively.
Therefore,
\begin{center}
$rank\,H_{i}(\Delta) = rank\,Z_{i}(\Delta) -\,rank\,B_{i} (\Delta)$.
\end{center}
One can easily see that rank $B_{d}(\Delta)=0$ due to $B_{d}(\Delta)=0$. For each $i\geqslant 0$ there is an exact sequence
\begin{center}
$0 \rightarrow Z_{i}(\Delta)
\rightarrow C_{i}(\Delta)^{
\underrightarrow{\partial_{i}}}B_{i-1}({\Delta})\rightarrow 0$.
\end{center}
Moreover,
\begin{center}
$f_{i}= rank\, C_{i}(\Delta)=rank\, Z_{i}({\Delta})+rank\,
 B_{i-1}(\Delta)$.
 \end{center}
 Therefore, the Euler characteristic of $\Delta$ can be expressed as\\
$\chi(\Delta) =  \sum_{i=0}^{d} (-1) ^{i}f_{i} = \sum_{i=0}^{d}(-1)^{i} (rank\,Z_{i}(\Delta) + rank\,B_{i-1}(\Delta))$\\
    = $\sum_{i=0}^{d}(-1)^{i} rank\,Z_{i}(\Delta) + \sum_{i=0}^{d}(-1)^{i}rank\,B_{i-1}(\Delta)$.\\
Changing index of summation in the last sum and using the fact that\\ rank $B_{-1}(\Delta)=0=rank\, B_{d}(\Delta)$, we get\\
 $\chi(\Delta) =  \sum_{i=0}^{d} (-1)^{i}\,rank  \,Z_{i} (\Delta) + \sum_{i=0}^{d} (-1)^{i+1}\,rank\,
 B_{i} (\Delta)$\\
  $=\sum_{i=0}^{d}(-1)^{i}\,(rank\,Z_{i}(\Delta) - rank\,
 B_{i}(\Delta))$
     $= \sum_{i=0}^{d}(-1)^{i}\,rank\,H_{i}(\Delta)$.\\
Thus, the Euler characteristic of $\Delta$ is given by
$$\chi(\Delta) =\sum_{i=0}^{d}(-1)^{i}\,\beta_{i}(\Delta),$$
where $\beta_{i}(\Delta)=rank\,H_{i}(\Delta)$ is the $i$-th Betti number of $\Delta$, see \cite{H} and \cite{R}.

\section{Topological Characterizations of $\Delta_s(\mathcal{U}_{n,m}^r)$}
Let
$\mathcal{U}_{n,m}^r$ be a uni-cyclic multigraph having $n$
edges including $r$ multiple edges within and outside the cycle of length $m$. We fix the
labeling of the edge set $E$ of $\mathcal{U}_{n,m}^r$ as follows:\\
$E=\{e_{11},\ldots,e_{1t_1},\ldots,e_{r'1},\ldots,e_{r't_{r'}},
e_{(r'+1)1},\ldots,e_{m1}, e_{(m+1)1},\ldots,e_{(m+1)t_{m+1}},\\
\ldots, e_{(m+r'')1},\ldots,e_{(m+r'')t_{m+r''}},e_1, \ldots, e_v\}$,
where $e_{i1},\ldots,e_{it_i}$ are the multiple edges of $i$-th edge
of cycle with $1\leq i\leq r'$ while $e_{(r'+1)1},\ldots,e_{m1}$ are
the single edges of the cycle and $e_{j1},\ldots,e_{jt_j}$ are the
multiple edges of the $j$-th edge outside the cycle with $m+1\leq
j\leq m+r''$, moreover, $e_1, \ldots, e_v$ are single edges appeared
outside the cycle.

We give first the characterization of $s(\mathcal{U}_{n,m}^r)$.

\begin{Lemma}\label{l1}
Let $\mathcal{U}_{n,m}^r$ be the uni-cyclic multigraph having $n$
edges including $r$ multiple edges and a cycle of length $m$ with the edge set $E$, given above.
A subset $E(T_{wi_w})\subset E$ will belong to $s(\mathcal{U}_{n,m}^r)$ if and only if
$T_{wi_w}=
\{e_{1i_1},\ldots,e_{r'i_{r'}},\\e_{(r'+1)1},\ldots,e_{m1},e_{(m+1)i_{(m+1)}},\ldots,e_{(m+r'')i_{(m+r'')}},e_1,\ldots,e_v\}\backslash
\{e_{wi_w}\}$
for some $i_h\in \{1,\ldots,t_h\}$ with $1\leq h\leq r'$,
$m+1\leq h\leq m+r''$ and ${i_w}\in \{1,\ldots,t_w\}$ with
$1\leq w\leq r'$ or $i_w=1$ with
$r'+1\leq w\leq m$ for some $w=h$ and $i_w=i_h$ appeared in $T_{wi_w}$.
\end{Lemma}
\begin{proof}
 By cutting down
method \cite{FH}, the spanning trees of $\mathcal{U}_{n,m}^r$ can be
obtained by removing exactly $t_h-1$ edges from each multiple edge such that $1\leq h\leq r'$,
$m+1\leq h\leq m+r''$
and in addition, an edge from the resulting cycle need to be removed. Therefore, the spanning trees will be of the form
$T_{wi_w}=
\{e_{1i_1},\ldots,e_{r'i_{r'}},e_{(r'+1)1},\ldots,e_{m1},e_{(m+1)i_{(m+1)}},\ldots,e_{(m+r'')i_{(m+r'')}},e_1,\ldots,e_v\}\backslash
\{e_{wi_w}\}$
for some $i_h\in \{1,\ldots,t_h\}$ with $1\leq h\leq r'$,
$m+1\leq h\leq m+r''$ and ${i_w}\in \{1,\ldots,t_w\}$ with
$1\leq w\leq r'$ or $i_w=1$ with
$r'+1\leq w\leq m$ for some $w=h$ and $i_w=i_h$ appeared in $T_{wi_w}$.
\end{proof}

In the following result, we give the primary decomposition of facet ideal
$I_\mathcal{F}(\Delta_s(\mathcal{U}_{n,m}^r)$.

\begin{Proposition}\label{p1} \rm{Let
$\Delta_s(\mathcal{U}_{n,m}^r)$ be the spanning simplicial complex
of uni-cyclic multigraph $\mathcal{U}_{n,m}^r$
having $n$ edges including $r$ multiple edges within and outside the cycle of length $m$. Then,\\
$I_\mathcal{F}(\Delta_s(\mathcal{U}_{n,m}^r)=$ $$\ \
\bigg(\bigcap\limits_{1\leq a \leq v}(x_a)\bigg)\bigcap
\bigg(\bigcap\limits_{1\leq i\leq r'\,;\,r'+1\leq k\leq
m}(x_{i1},\ldots,x_{it_i},x_{k1})\bigg)\bigcap
\bigg(\bigcap\limits_{r'+1\leq k< l\leq m}(x_{k1},
x_{l1})\bigg)$$
$$ \ \bigcap \bigg(\bigcap\limits_{1\leq i < b\leq
r'}(x_{i1},\ldots,x_{it_i}, x_{b1},\ldots,x_{bt_b})\bigg)\bigcap
\bigg(\bigcap\limits_{m+1\leq j\leq
m+r''}(x_{j1},\ldots,x_{jt_j})\bigg),$$ where $t_i$ with $1\leq i\leq
r'$ is the number of multiple edges appeared in the $i$-th edge of
the cycle and $t_j$ with $m+1\leq j\leq m+r''$ is the number of
multiple edges appeared in the $j$th-edge outside the cycle}.
\end{Proposition}
\begin{proof}
Let $I_\mathcal{F}(\Delta_s(\mathcal{U}_{n,m}^r)$ be the facet ideal
of the spanning simplicial complex $\Delta_s(\mathcal{U}_{n,m}^r)$.
From (\cite{RHV}, Proposition 1.8), minimal prime ideals of the
facet ideal $I_\mathcal{F}(\Delta)$ have one-to-one correspondence
with the minimal vertex covers of the simplicial complex $\Delta$.
Therefore, in order to find the primary decomposition of the facet
ideal $I_\mathcal{F}(\Delta_s(\mathcal{U}_{n,m}^r))$; it is
sufficient to find all the minimal vertex covers of
$\Delta_s(\mathcal{U}_{n,m}^r)$.\\
\indent As $e_a$, $a\in [v]$ is not an edge of the cycle of
uni-cyclic multigraph $\mathcal{U}_{n,m}^r$ and does not belong to
any multiple edge of $\mathcal{U}_{n,m}^r$. Therefore, it is clear
by definition of minimal vertex cover that $\{e_a\}$, $1\leq a\leq
v$ is a minimal vertex cover of $\Delta_s(\mathcal{U}_{n,m}^r)$.
Moreover, a spanning tree is obtained by removing exactly ${t_h}-1$
edges from each multiple edge with $h=1,\ldots,r',m+1,\ldots,m+r''$ and in addition, an edge from
the resulting cycle of $\mathcal{U}_{n,m}^r$.
We illustrate the result into the following cases.\\
{\bf Case 1.}
If atleast one multiple edge is appeared in the cycle of $\mathcal{U}_{n,m}^r$. Then, we cannot remove one complete multiple edge and one single edge from the cycle
of $\mathcal{U}_{n,m}^r$ to get spanning tree. Therefore,
$(x_{i1},\ldots,x_{it_i}, x_{k1})$ with $1\leq i\leq r'$, $r'+1\leq
k\leq m$ is a minimal vertex cover of the spanning simplicial
complex $\Delta_s(\mathcal{U}_{n,m}^r)$ having non-empty
intersection with all the spanning trees of $\mathcal{U}_{n,m}^r$.
Moreover, two single edges cannot be removed from the cycle of
$\mathcal{U}_{n,m}^r$ to get spanning tree.
Consequently, $(x_{k1}, x_{l1})$ for $r'+1\leq k < l\leq m$ is a
minimal vertex cover of $\Delta_s(\mathcal{U}_{n,m}^r)$
having non-empty intersection with all the spanning trees of $\mathcal{U}_{n,m}^r$.\\
{\bf Case 2.}
If atleast two multiple edges are appeared in the cycle of $\mathcal{U}_{n,m}^r$. Then, two complete multiple edges cannot be
removed from the cycle of $\mathcal{U}_{n,m}^r$ to get spanning
tree. Consequently,
$(x_{i1},\ldots,x_{it_i},x_{b1},\ldots,x_{bt_b})$ for $1\leq i<
b\leq r'$ is a minimal vertex cover of $\Delta_s(\mathcal{U}_{n,m}^r)$ having non-empty
intersection with all the spanning trees of $\mathcal{U}_{n,m}^r$.\\
{\bf Case 3.}
If atleast one multiple edge appeared outside the cycle of
$\mathcal{U}_{n,m}^r$. Then, one complete multiple edge outside the cycle of $\mathcal{U}_{n,m}^r$ cannot
be removed to get spanning tree. So, $(x_{j1},\ldots,x_{jt_j})$ for
$m+1\leq j\leq m+r''$ is a minimal vertex cover of $\Delta_s(\mathcal{U}_{n,m}^r)$ having non-empty
intersection with all the spanning trees of
$\mathcal{U}_{n,m}^r$. This completes the proof.
\end{proof}

We give now formula for Euler characteristic of $\Delta_s(\mathcal{U}_{n,m}^r)$.

\begin{Theorem}\label{t1}
\rm{Let $\Delta_s(\mathcal{U}_{n,m}^r)$ be spanning simplicial
complex of uni-cyclic multigraph $\mathcal{U}_{n,m}^r$
having $n$ edges including $r$ multiple edges and a cycle of length
$m$. Then, $dim(\Delta_s(\mathcal{U}_{n,m}^r))=
n-\sum\limits_{i=1}^{r'}{t_i}-\sum\limits_{j=m+1}^{m+r''}{t_j}+r-2$
and the Euler characteristic of $\Delta_s(\mathcal{U}_{n,m}^r)$ is given by\\
$\chi(\Delta_s(\mathcal{U}_{n,m}^r))=\sum\limits_{i=0}^{n-1}(-1)^i\bigg[{n\choose
i+1}-\prod\limits_{i=1}^{r'}{t_i}\bigg[{{n-\alpha+r'-m} \choose
{i+1-m}}\\-\sum\limits_{j=2}^{\beta}\bigg({\beta\choose
j}-{\bigg(\sum\limits_{m+1\leq {i_1}<\ldots<{i_j\leq
m+r''}}}\prod\limits_{k=i_1}^{i_j}{{t_k}\choose
1}\bigg)\bigg)\sum\limits_{l=j}^{\beta}(-1)^{l-j}
{{\beta-j}\choose{l-j}}{{n-\alpha+r'-m-l} \choose
{i+1-m-l}}\bigg]$\\
$-\sum\limits_{j=2}^{\alpha+\beta}\bigg({\alpha+\beta\choose
j}-{\bigg(\sum\limits^{{i_j\notin}\{r'+1,\ldots,m\}}_{1\leq
{i_1}<\ldots<{i_j\leq
m+r''}}}\prod\limits_{k=i_1}^{i_j}{{t_k}\choose
1}\bigg)\bigg)\sum\limits_{l=j}^{\alpha+\beta}(-1)^{l-j}
{{\alpha+\beta-j}\choose{l-j}}{{n-l} \choose {i+1-l}}\bigg]$\\
with $\alpha=\sum\limits_{i=1}^{r'}{t_i}$ and $\beta=\sum\limits_{j=m+1}^{m+r''}{t_j}$ such that $r'$ and $r''$ are the number of multiple edges appeared within and outside the
cycle, respectively.}
\end{Theorem}
\begin{proof}Let
$E=\{e_{11},\ldots,e_{1t_1},\ldots,e_{r'1},\ldots,e_{r't_{r'}},
e_{(r'+1)1},\ldots,e_{m1}, e_{(m+1)1},\ldots,\\e_{(m+1)t_{m+1}},\ldots,
e_{(m+r'')1},\ldots,e_{(m+r'')t_{m+r''}},e_1, \ldots, e_v\}$ be the
edge set of uni-cyclic multigraph $\mathcal{U}_{n,m}^r$ having $n$
edges including $r$ multiple edges and a cycle of length $m$ such that
$r'$ and $r''$ are the number of multiple edges appeared within and outside the
cycle, respectively.

One can easily see that each facet $T_{wi_w}$ of $\Delta_s(\mathcal{U}_{n,m}^r)$ is of the same dimension
$n-\sum\limits_{i=1}^{r'}{t_i}-\sum\limits_{j=m+1}^{m+r''}{t_j}+r-2=n-\alpha-\beta+r-2$ with
$\alpha=\sum\limits_{i=1}^{r'}{t_i}$ and $\beta=\sum\limits_{j=m+1}^{m+r''}{t_j}$, see Lemma \ref{l1}.

By definition, $f_i$ is the number of subsets of $E$ with $i+1$ elements not containing cycle and multiple edges.
There are $\prod\limits_{i=1}^{r'}{t_i}{{n-\alpha+r'-m} \choose {i+1-m}}$
number of subsets of $E$ containing cycle but not containing any
multiple edge within the cycle. There are
$${{n-\alpha+r'-m-\beta} \choose {i+1-m-\beta}}={\beta\choose \beta}\sum\limits_{l=\beta}^{\beta}(-1)^{l-\beta}
{{\beta-\beta}\choose{l-\beta}}{{n-\alpha+r'-m-l} \choose {i+1-m-l}}$$
subsets of $E$ containing cycle and $\beta$ multiple edges of
$\mathcal{U}_{n,m}^r$ outside the cycle but not containing any
multiple edge within the cycle. There are
$${\beta\choose \beta-1}\bigg[{{n-\alpha+r'-m-(\beta-1)}\choose {i+1-m-(\beta-1)}}-{{n-\alpha+r'-m-\beta}\choose
{i+1-m-\beta}}\bigg]$$
$$={\beta\choose \beta-1}\sum\limits_{l=\beta-1}^{\beta}(-1)^{l-(\beta-1)}{{\beta-(\beta-1)}\choose{l-(\beta-1)}}{{n-\alpha+r'-m-l}\choose {i+1-m-l}}$$
subsets of $E$ containing cycle and $\beta-1$ multiple edges of
$\mathcal{U}_{n,m}^r$ outside the cycle but not containing any
multiple edge within the cycle. Continuing in similar manner, the number of subsets
of $E$ containing cycle and two edges from a multiple edge outside the cycle but not containing any multiple edge within the cycle is given by\\
${{\beta-2}\choose 0}{{n-\alpha+r'-m-2}\choose
{i+1-m-2}}-{{\beta-2}\choose 1}{{n-\alpha+r'-m-3}\choose
{i+1-m-3}}+{{\beta-2}\choose 2}{{n-\alpha+r'-m-4}\choose
{i+1-m-4}}+\cdots
\\+ (-1)^{\beta-2}{{\beta-2}\choose {\beta-2}}{{n-\alpha+r'-m-\beta}\choose
{i+1-m-\beta}}=\sum\limits_{l=2}^{\beta}(-1)^{l-2}
{{\beta-2}\choose{l-2}}{{n-\alpha+r'-m-l}\choose{i+1-m-l}}.$ \\There
are $\bigg({\beta\choose 2}-{\bigg(\sum\limits_{m+1\leq {i_1}<{i_2}\leq
m+r''}}\prod\limits_{k=i_1}^{i_2}{{t_k}\choose 1}\bigg)\bigg)$ choices of
two edges from a multiple edge outside the cycle. Therefore,
we obtain
$$\bigg({\beta\choose 2}-{\bigg(\sum\limits_{m+1\leq {i_1}<{i_2}\leq
m+r''}}\prod\limits_{k=i_1}^{i_2}{{t_k}\choose
1}\bigg)\bigg)\sum\limits_{l=2}^{\beta}(-1)^{l-2}
{{\beta-2}\choose{l-2}}{{n-\alpha+r'-m-l}\choose{i+1-m-l}}$$
the number of subsets of $E$ with $i+1$ elements containing cycle and all possible choices of two edges from
a multiple edge outside the cycle but not containing multiple edges within the cycle.

Now, we use inclusion exclusion principal to obtain \\
(number of subsets of $E$ containing cycle but not containing
multiple edges)= $\prod\limits_{i=1}^{r'}{t_i}$[(number of subsets
of $E$ with $i+1$ elements containing cycle but not containing
multiple edges within the cycle)$-$(number of subsets of $E$ with $i+1$
elements containing cycle and $\beta$ multiple edges outside the
cycle but not containing multiple edges within the cycle)$-$(number of
subsets of $E$ with $i+1$ elements containing cycle and $\beta-1$
multiple edges outside the cycle but not containing multiple edges
within the cycle)$-\cdots-$(number of subsets of $E$ with $i+1$ elements
containing cycle and two edges from a multiple edge
outside the cycle but not containing multiple edges within
the cycle)]\\
$= \prod\limits_{i=1}^{r'}{t_i}\bigg[{{n-\alpha+r'-m} \choose
{i+1-m}}-{\beta\choose
\beta}\sum\limits_{l=\beta}^{\beta}(-1)^{l-\beta}
{{\beta-\beta}\choose{l-\beta}}{{n-\alpha+r'-m-l} \choose
{i+1-m-l}}$\\ $- {\beta\choose
\beta-1}\sum\limits_{l=\beta-1}^{\beta}(-1)^{l-(\beta-1)}{{\beta-(\beta-1)}\choose{l-(\beta-1)}}{{n-\alpha+r'-m-l}\choose
{i+1-m-l}}$
\\ $-\cdots-\bigg({\beta\choose 2}-{\bigg(\sum\limits_{m+1\leq {i_1}<{i_2}\leq
m+r''}}\prod\limits_{k=i_1}^{i_2}{{t_k}\choose
1}\bigg)\bigg)\sum\limits_{l=2}^{\beta}(-1)^{l-2}
{{\beta-2}\choose{l-2}}{{n-\alpha+r'-m-l}\choose{i+1-m-l}}\bigg]$\\
$=\prod\limits_{i=1}^{r'}{t_i}\bigg[{{n-\alpha+r'-m} \choose
{i+1-m}}\\-\sum\limits_{j=2}^{\beta}\bigg({\beta\choose
j}-{\bigg(\sum\limits_{m+1\leq {i_1}<\ldots<{i_j\leq
m+r''}}}\prod\limits_{k=i_1}^{i_j}{{t_k}\choose
1}\bigg)\bigg)\sum\limits_{l=j}^{\beta}(-1)^{l-j}
{{\beta-j}\choose{l-j}}{{n-\alpha+r'-m-l} \choose
{i+1-m-l}}\bigg]$.

Therefore, we compute
\\$f_i=$ (number of subsets of $E$ with $i+1$ elements)$-$(number of
subsets of $E$ with $i+1$ elements containing cycle but not
containing multiple edges)$-$(number of subsets of $E$ with $i+1$
elements containing $\alpha+\beta$ multiple edges)$-$(number of
subsets of $E$ with $i+1$ elements containing $\alpha+\beta-1$
multiple edges)$-\cdots-$(number of subsets of $E$ with $i+1$
elements
containing two edges from a multiple edge of $\mathcal{U}_{n,m}^r$)\\
$={n\choose i+1}-\prod\limits_{i=1}^{r'}{t_i}\bigg[{{n-\alpha+r'-m}
\choose {i+1-m}}\\-\sum\limits_{j=2}^{\beta}\bigg({\beta\choose
j}-{\bigg(\sum\limits_{m+1\leq {i_1}<\ldots<{i_j\leq
m+r''}}}\prod\limits_{k=i_1}^{i_j}{{t_k}\choose
1}\bigg)\bigg)\sum\limits_{l=j}^{\beta}(-1)^{l-j}
{{\beta-j}\choose{l-j}}{{n-\alpha+r'-m-l} \choose
{i+1-m-l}}\bigg]$\\
$-{\alpha+\beta\choose
\alpha+\beta}\sum\limits_{l=\alpha+\beta}^{\alpha+\beta}(-1)^{l-(\alpha+\beta)}
{{{(\alpha+\beta)}-{(\alpha+\beta)}}\choose{l-{(\alpha+\beta)}}}{{n-l}
\choose {i+1-l}}$\\ $- {\alpha+\beta\choose
\alpha+\beta-1}\sum\limits_{l=\alpha+\beta-1}^{\alpha+\beta}(-1)^{l-(\alpha+\beta-1)}
{{(\alpha+\beta)-(\alpha+\beta-1)}\choose{l-(\alpha+\beta-1)}}{{n-l}\choose
{i+1-l}}$
\\ $-\cdots-\bigg({\alpha+\beta\choose 2}-{\bigg(\sum\limits^{{i_j\notin}\{r'+1,\ldots,m\}}_{1\leq {i_1}<{i_2}\leq
m+r''}}\prod\limits_{k=i_1}^{i_2}{{t_k}\choose
1}\bigg)\bigg)\sum\limits_{l=2}^{\alpha+\beta}(-1)^{l-2}
{{\alpha+\beta-2}\choose{l-2}}{{n-l}\choose{i+1-l}}$\\
$={n\choose i+1}-\prod\limits_{i=1}^{r'}{t_i}\bigg[{{n-\alpha+r'-m}
\choose {i+1-m}}\\-\sum\limits_{j=2}^{\beta}\bigg({\beta\choose
j}-{\bigg(\sum\limits_{m+1\leq {i_1}<\ldots<{i_j\leq
m+r''}}}\prod\limits_{k=i_1}^{i_j}{{t_k}\choose
1}\bigg)\bigg)\sum\limits_{l=j}^{\beta}(-1)^{l-j}
{{\beta-j}\choose{l-j}}{{n-\alpha+r'-m-l} \choose
{i+1-m-l}}\bigg]$ \\
$-\sum\limits_{j=2}^{\alpha+\beta}\bigg({\alpha+\beta\choose
j}-{\bigg(\sum\limits^{{i_j\notin}\{r'+1,\ldots,m\}}_{1\leq
{i_1}<\ldots<{i_j\leq
m+r''}}}\prod\limits_{k=i_1}^{i_j}{{t_k}\choose
1}\bigg)\bigg)\sum\limits_{l=j}^{\alpha+\beta}(-1)^{l-j}
{{\alpha+\beta-j}\choose{l-j}}{{n-l} \choose {i+1-l}}$.
\end{proof}

\begin{figure}[htbp]
        \centerline{\includegraphics[width=10.0cm]{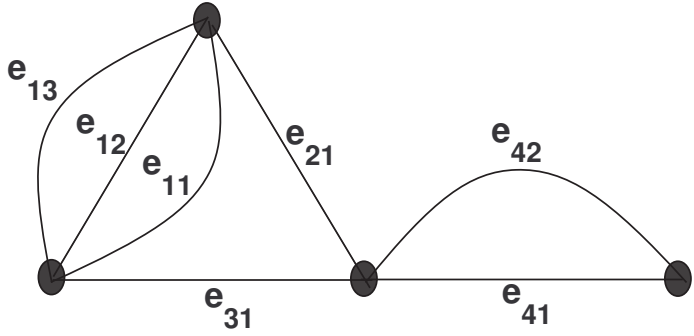}}
        \caption{$\mathcal{U}_{7,3}^2$}
        \label{fig:f1}
\end{figure}

\begin{Example}
\rm{Let $E=\{e_{11}, e_{12}, e_{13}, e_{21}, e_{31}, e_{41},
e_{42}\}$ be the edge set of uni-cyclic multigraph
$\mathcal{U}_{7,3}^2$ having $7$ edges including $2$ multiple edges
and a cycle
of length $3$, as shown in Figure ~\ref{fig:f1}. By cutting-down method, we obtain\\
$s(\mathcal{U}_{7,3}^2)=\{ \{e_{21},e_{31},e_{41}\},
\{e_{11},e_{31},e_{41}\}, \{e_{11},e_{21},e_{41}\},
\{e_{12},e_{31},e_{41}\}, \{e_{12},e_{21},e_{41}\},\\
\{e_{13},e_{31},e_{41}\}, \{e_{13},e_{21},e_{41}\},
\{e_{21},e_{31},e_{42}\}, \{e_{11},e_{31},e_{42}\},
\{e_{11},e_{21},e_{42}\}, \{e_{12},e_{31},e_{42}\},\\
\{e_{12},e_{21},e_{42}\}, \{e_{13},e_{31},e_{42}\},
\{e_{13},e_{21},e_{42}\}\}$.\\
By definition, $f_i$ is the number of subsets of $E$ with
$i+1$ elements not containing cycle and multiple edges. Since,
$\{e_{11}\}, \{e_{12}\}, \{e_{13}\}, \{e_{21}\}, \{e_{31}\},
\{e_{41}\}$\\ and $\{e_{42}\}$ are subsets of $E$ containing one
element. It implies that $f_0=7$. There are $\{e_{11},e_{21}\},
\{e_{11},e_{31}\}, \{e_{11},e_{41}\}, \{e_{11},e_{42}\},
\{e_{12},e_{21}\}, \{e_{12},e_{31}\}, \{e_{12},e_{41}\},\\
\{e_{12},e_{42}\}, \{e_{13},e_{21}\}, \{e_{13},e_{31}\},
\{e_{13},e_{41}\}, \{e_{13},e_{42}\}, \{e_{21},e_{31}\},
\{e_{21},e_{41}\}, \{e_{21}, e_{42}\},\\
\{e_{31},e_{41}\}, \{e_{31}, e_{42}\}$ subsets of $E$ containing two
elements but not containing cycle and multiple edges. So, $f_1= 17$.
We know that the spanning trees of $\mathcal{U}_{7,3}^2$ are
2-dimensional facets of the spanning simplicial complex
$\Delta_s(\mathcal{U}_{7,3}^2)$. Therefore, $f_2= 14$. Thus,
$$\chi(\Delta_s(\mathcal{U}_{7,3}^2))= f_0 - f_1 + f_2=7 - 17 + 14 = 4.$$
Now, we compute the Euler characteristic of $\Delta_s(\mathcal{U}_{7,3}^2)$ by using Theorem ~\ref{t1}.
We observe that, $n=7$, $r'=1$, $r''=1$, $r=2$, $m=3$, $\alpha=3$,
$\beta=2$ and $0\leq i\leq d$, where $d=n-\alpha-\beta+r-2=2$ is the
dimension of
$\Delta_s(\mathcal{U}_{7,3}^2)$.\\
By substituting these values in Theorem ~\ref{t1}, we get\\
$f_i={7\choose i+1}-3\bigg[{{2} \choose {i+1-3}}-{2\choose 2}{{0}
\choose {i+1-5}}\bigg]-[{5\choose 2}-{{3\choose 1}{2\choose
1}}]\bigg({3\choose0}{5 \choose {i+1-2}}-{3\choose1}{4 \choose
{i+1-3}}+{3\choose2}{3 \choose {i+1-4}}-{3\choose3}{2 \choose
{i+1-4}}\bigg)-{5\choose 3}{\bigg({2\choose0}{4 \choose
{i+1-3}}-{2\choose1}{3 \choose {i+1-4}}}+{2\choose2}{2 \choose
{i+1-5}}\bigg)-\\{5\choose 4}\bigg({1\choose0}{3 \choose
{i+1-4}}-{1\choose1}{2 \choose {i+1-5}}\bigg)-{5\choose 5}{2 \choose
{i+1-5}}$.
\\
Alternatively, we compute $(f_0,f_1,f_2)=(7, 17, 14)$
and
$\chi(\Delta_s(\mathcal{U}_{7,3}^2))=4$.

We compute now the Betti numbers of $\Delta_s(\mathcal{U}_{7,3}^2)$. The facet ideal of $\Delta_s(\mathcal{U}_{7,3}^2)$ is given by\\
$I_\mathcal{F}(\Delta_s(\mathcal{U}_{7,3}^2))=\langle
x_{F_{21,31,41}},x_{F_{11,31,41}},x_{F_{11,21,41}},x_{F_{12,31,41}},x_{F_{12,21,41}}
,x_{F_{13,31,41}},x_{F_{13,21,41}},\\
x_{F_{21,31,42}},x_{F_{11,31,42}},x_{F_{11,21,42}},x_{F_{12,31,42}},x_{F_{12,21,42}},x_{F_{13,31,42}},x_{F_{13,21,42}}\rangle$.\\
We consider the chain complex of
$\Delta_s(\mathcal{U}_{7,3}^2)$
$$0\ \underrightarrow {\partial_3} \ C_2 \ \underrightarrow
{\partial_2} \ C_1 \ \underrightarrow {\partial_1} \ C_0
 \ \underrightarrow {\partial_0} \ 0$$
The homology groups are given by
$H_i=\frac{Ker(\partial_i)}{Im(\partial_{i+1})}$ with $i=0,1,2.$\\
Therefore, the $i$-th Betti number of
$\Delta_s(\mathcal{U}_{7,3}^2)$ is given by
$$\beta_i=rank(H_i)=rank(Ker(\partial_i))-rank(Im(\partial_{i+1})) \ \rm{with} \ i=0,1,2.$$
Now, we compute rank and nullity of the matrix $\partial_i$ of order
$f_i\times f_{i+1}$ with $i=0,1,2$.\\
The boundary homomorphism
$\partial_2:C_2(\Delta_s(\mathcal{U}_{7,3}^2))\rightarrow
C_1(\Delta_s(\mathcal{U}_{7,3}^2))$ can be expressed as\\
$\partial_2(F_{21,31,41})=F_{31,41}-F_{21,41}+F_{21,31}$;
$\partial_2(F_{11,31,41})=F_{31,41}-F_{11,41}+F_{11,31}$;
$\partial_2(F_{11,21,41})=F_{21,41}-F_{11,41}+F_{11,21}$;
$\partial_2(F_{12,31,41})=F_{31,41}-F_{12,41}+F_{12,31}$;
$\partial_2(F_{12,21,41})=F_{21,41}-F_{12,41}+F_{12,21}$;
$\partial_2(F_{13,31,41})=F_{31,41}-F_{13,41}+F_{13,31}$;
$\partial_2(F_{13,21,41})=F_{21,41}-F_{13,41}+F_{13,21}$;
$\partial_2(F_{21,31,42})=F_{31,42}-F_{21,42}+F_{21,31}$;
$\partial_2(F_{11,31,42})=F_{31,42}-F_{11,42}+F_{11,31}$;
$\partial_2(F_{11,21,42})=F_{21,42}-F_{11,42}+F_{11,21}$;
$\partial_2(F_{12,31,42})=F_{31,42} - F_{12,42} + F_{12,31}$;
$\partial_2(F_{12,21,42})=F_{21,42} - F_{12,42} + F_{12,21}$;
$\partial_2(F_{13,31,42})=F_{31,42}-F_{13,42}+F_{13,31}$;
$\partial_2(F_{13,21,42})=F_{21,42}-F_{13,42}+F_{13,21}$.\\
The
boundary homomorphism
$\partial_1:C_1(\Delta_s(\mathcal{U}_{7,3}^2))\rightarrow
C_0(\Delta_s(\mathcal{UM}_{7,3}^2))$ can be written as\\
$\partial_1(F_{11,21})=e_{21}-e_{11}$;
$\partial_1(F_{11,31})=e_{31}-e_{11}$;
$\partial_1(F_{11,41})=e_{41}-e_{11}$;
$\partial_1(F_{11,42})=e_{42}-e_{11}$;
$\partial_1(F_{12,21})=e_{21}-e_{12}$;
$\partial_1(F_{12,31})=e_{31}-e_{12}$;
$\partial_1(F_{12,41})=e_{41}-e_{12}$;
$\partial_1(F_{12,42})=e_{42}-e_{12}$;
$\partial_1(F_{13,21})=e_{21}-e_{13}$;
$\partial_1(F_{13,31})=e_{31}-e_{13}$;
$\partial_1(F_{13,41})=e_{41}-e_{13}$;
$\partial_1(F_{13,42})=e_{42}-e_{13}$;
$\partial_1(F_{21,31})=e_{31}-e_{21}$;
$\partial_1(F_{21,41})=e_{41}-e_{21}$;
$\partial_1(F_{21,42})=e_{42}-e_{21}$;
$\partial_1(F_{31,41})=e_{41}-e_{31}$;
$\partial_1(F_{31,42})=e_{42}-e_{31}$.\\ Then, by using MATLAB, we
compute rank of $\partial_2=11$; nullity of $\partial_2=3$; rank of
$\partial_1=6$; nullity of $\partial_1=11$.\\ Therefore, the Betti
numbers are given by
$\beta_0=rank(Ker(\partial_0))-rank(Im(\partial_{1}))=7-6=1$;
$\beta_1=rank(Ker(\partial_1))-rank(Im(\partial_{2}))=11-11=0$;
$\beta_2=rank(Ker(\partial_2))-rank(Im(\partial_{3}))=3-0=3$.\\
Alternatively, the Euler characteristic of $\Delta_s(\mathcal{U}_{7,3}^2)$ is given by
$$\chi(\Delta_s(\mathcal{U}_{7,3}^2))=\beta_0-\beta_1+\beta_2=1-0+3=4.$$}
\end{Example}

\end{document}